\documentclass[11pt]{amsart}

\newtheorem{theorem}{Theorem}[section]
\newtheorem{lemma}[theorem]{Lemma}

\theoremstyle{definition}

\newtheorem{corollary}[theorem]{Corollary}

\theoremstyle{remark}
\newtheorem{remark}[theorem]{Remark}

\numberwithin{equation}{section}

\newcommand{\dd}{\mathrm d}
\newcommand{\grad}{\mathrm{grad}\ }
\newcommand{\tr}{\mathrm{tr}\,}
\newcommand{\hs}{\mathrm{Hess}\,}
\newcommand{\M}{\widetilde M}
\newcommand{\F}{\mathcal F}
\newcommand{\Ver}{\mathcal V}
\newcommand{\Hor}{\mathcal H}
\newcommand{\hf}{H^{\mathcal F}}
\newcommand{\D}{\mathcal D}
\newcommand{\Nt}{\widetilde\nabla}
\newcommand{\X}{\mathfrak X}
\newcommand{\Xt}{\widetilde X}

\newcommand{\pie}{\langle}
\newcommand{\pid}{\rangle}
\newcommand{\til}{\widetilde}
\newcommand{\ov}{\overline}

\begin{document}

\title{On the fundamental tone of immersions and submersions}

\author[Cavalcante]{Marcos P. Cavalcante}     
\address{IM, Universidade Fe\-deral de Alagoas, Macei\'o, 
AL, CEP 57072-970, Brazil}
\email{marcos@pos.mat.ufal.br}

\author[Manfio]{Fernando Manfio}
\address{ICMC, Universidade de S\~ao Paulo, 
S\~ao Carlos, SP, CEP 13561-060 , Brazil} 
\email{manfio@icmc.usp.br}

\date{\today}

\begin{abstract}
In this paper we obtain  lower bound estimates of the spectrum 
of Laplace-Beltrami operator on complete submanifolds 
with bounded mean curvature, whose ambient space admits 
a Riemannian submersion over a  Riemannian 
manifold with  negative sectional curvature. 
Our main theorem generalizes many previous known estimates and applies
for both immersions and submersions.
\end{abstract}

\maketitle


\section{Introduction}

Given a 
 compact domain $\Omega$ on a $m$-dimensional Riemannian manifold $M^m$
let us denote by $Spec (\Omega)=\{\lambda_1(\Omega)<\lambda_2(\Omega)\leq ...\}$ the set of eigenvalues 
of the Laplace-Beltrami operator $-\Delta$ on $\Omega$ with Dirichlet  boundary condition, repeated according to its multiplicity.
That is, for each $\lambda_i=\lambda_i(\Omega)$, $i=1, 2, \ldots$ there exists a nontrival solution to the following problem.
\[
\left\{\begin{array}{rcc}
-\Delta\varphi &=&\lambda_i\varphi \quad \text{ in }\ \Omega, \\
\varphi &=&0  \quad \text{ on }  \partial \Omega. \\
\end{array}\right.
\]

The study of the relations between the eigenvalues and the geometry of the domain (or the manifold) 
is a very active topic on differential geometry and has been attracted the attention of 
many pure and applied mathematicians from a long time. 

In this paper, we are interested in obtaining lower bound estimates of 
the spectrum of Laplacian on a class of complete noncompact Riemannian
manifolds in terms of its geometry. In order to state our results we need
some notations. 

We first recall that on compact domains, the set of eigenvalues is the 
whole spectrum of the Laplace-Beltrami operator for the Dirichlet problem.
When we deal with noncompact domains some accumulation points or 
eigenvalues of infinite multiplicity may appear, composing the essential
spectrum. In any case, the bottom of the spectrum is given by a limit of
the first eigenvalues when we consider an exhaustion of the domain. 
More precisely,  if $M$ is a Riemannian manifold and 
$\Omega_1\subset \Omega_2\subset\cdots$ is an exhaustion of $M$
the {\em fundamental tone} of $M$ is defined by 
$\lambda_1(M)=\lim_{k\to \infty}\lambda_1(\Omega_k)$.

Of course it does  not depend on the choice of the exhaustion and 
 coincides with the first  eigenvalue when $M$ is compact.
Moreover,  $\lambda_1(M)$ can be characterized variationally as following: 
$$
\displaystyle \lambda_1(M) = \inf \bigg\{ \frac{-\int_M \varphi \Delta \varphi}
{\int_M \varphi^2}: \forall \varphi\in C_0^\infty(M) \bigg \}.
$$
In particular, $\lambda_1(M)\geq 0$ and it is the bottom of the spectrum of $-\Delta$ on $M$.

Clearly it is much harder to give a lower bound for $\lambda_1(M)$ than an upper bound,
and an important question that is proposed is to find conditions
on $M$ which imply $\lambda_1(M)>0$ (see \cite[$\S$III.4]{yau}).

In this direction, one of important contribution  were 
done by McKean \cite{mckean}, who proved that if $M$ is 
simply connected and its sectional curvature satisfies 
$K_M\leq -1$, then 
$$
\lambda_1(M)\geq  \frac{(m-1)^2}{4}= \lambda_1(\mathbb{H}^m),
$$ 
where $\mathbb{H}^m$ denotes the $m$-dimensional hyperbolic 
space of sectional curvature $-1$. 
This estimate was extended by  Veeravalli 
\cite{veeravalli}  for a quite 
general class of manifolds.

In the context of submanifolds, Cheung and Leung \cite{cheung}
gave lower bounds estimates when $M$ is complete and isometrically
immersed in the hyperbolic space $\mathbb{H}^n$ with bounded 
mean curvature vector field $\|H\|\leq \alpha <m-1$. Namely they 
proved that
\[
\lambda_1(M)\geq \frac{(m-1-\alpha)^2}{4}.
\]

Later, Bessa and Montenegro (see \cite[Corollary 4.4]{BM}) 
generalized Cheung-Leung's estimated
for the case where $M^m$ is immersed in a complete simply 
connected Riemannian manifold $\overline M^n$ with bounded 
sectional curvature $K_{\overline M} \leq -b^2<0$ and bounded mean
curvature vector $H$, with $\|H\| \leq \alpha<(m-1)b$. In this setting, 
they were able to prove that 
\[
\lambda_1(M)\geq \frac{[(m-1)b-\alpha]^2}{4}.
\]
We point out that  Castillon obtained a different lower bound estimate
in the same situation (see Th\'eor\`eme 2.3 in \cite{castillon}).

Fill years ago, B\'erard, Castillon and the first author \cite{cavalcante}, 
using a different 
approach, obtained  a sharp lower bound estimate 
for $\lambda_1(M)$, when $M$ is a hypersurface immersed into 
$\mathbb{H}^n\times\mathbb{R}$ with constant  mean curvature.

Our first result is a dual estimate of Cheung and Leung's theorem  
in the context of Riemannian submersions. We obtain the following

\begin{theorem}\label{teo:main1}
Let $\pi:M^m\to\mathbb{H}^k$ be a Riemannian submersion of a
complete Riemannian manifold $M^m$ onto the hyperbolic space. 
Let us denote by $H^\F$ the mean curvature of its fibers and 
assume that $\|H^\F\|\leq \beta<k-1$. Then
\begin{eqnarray*}
\lambda_1(M)\geq\frac{(k-1-\beta)^2}{4}.
\end{eqnarray*}
\end{theorem}

Notice that this estimate is sharp in the sense that  it is archived by
the canonical (totally geodesic) submersion of 
$\mathbb H^k\times \mathbb R^{m-k}$ over $\mathbb H^k$.

We have a similar estimate for submersions over a complete 
Riemannian manifold with sectional curvature bounded from 
above by a negative constant, and thus we also get  the dual
result of Bessa and Montenegro (it is a direct corollary of
Theorem \ref{teo:main2} below).

In fact, we found a general lower bound for $\lambda_1(M)$ 
for complete submanifolds with bounded mean curvature, 
whose ambient space admits a Riemannian submersion over
a complete Riemannian manifold with bounded negative 
sectional curvature. In particular, when the base manifold of the 
submersion is the hyperbolic space $\mathbb{H}^k$, our main
theorem reads as follows.

\begin{theorem}\label{teo:main}
Let $f:M^m\to\M^n$ be an isometric immersion of a complete 
Riemannian manifold $M^m$ into a Riemannian manifold 
$\M^n$, which admits a Riemannian submersion $\pi:\M\to \mathbb{H}^k$.
Let $H$ be the mean curvature of $M$, $\alpha^\F$ the second fundamental
form of the fibers of $\M$, $H^\F$ its mean curvature and
$A$ the O'Neill tensor of $\M$. If 
$$
c = \inf \{ k-1- \| H\| - \|H^\F \| - (n-m)\big( 2\|A\|_\infty+\|\alpha^\F\|_\infty +1 \big) \}>0,
$$ 
then
\begin{eqnarray*}
\lambda_1(M)\geq\frac{c^2}{4},
\end{eqnarray*}
where $\|A\|_\infty$ and $\|\alpha^\mathcal F\|_\infty$ denote the uniform norm of these tensors.
\end{theorem}

Note that we get the Theorem \ref{teo:main1} when the immersion 
$f$ is the identity, and we get (a new proof of)  Cheung-Leung's
Theorem when the submersion $\pi$ is the identity. In fact, in the 
former case $n=m$ and $H=0$ and in the latter case 
$\|H^\F \| = \| A\|_\infty = \|\alpha^\F\|_\infty=0.$

The paper is organized as follows. In Section \ref{pre}, we recall
some basic properties, in particular an useful condition on a 
Riemannian manifold which implies a positive lower bound 
estimate for the first eigenvalue.
In Sections \ref{sub} and \ref{busemann} we present some
results on Riemannian submersions and on Busemann functions.
A main step in our approach is to use a comparison theorem
for the Hessian of Busemann functions. Finally, in Section \ref{main},
we state and prove our general theorem (Theorem \ref{teo:main2}),
which generalizes Theorem \ref{teo:main} in two directions: 
when the base manifold has bounded negative sectional curvature
and when the base manifold is a Riemannian warped-product of a 
complete manifold by the real line. We also describe some 
examples of submersions where the constant in the main 
theorem is positive.

\vspace{.2cm}

The authors are grateful to Professors P. Piccione, H. Rosenberg
and D. Zhou for helpful comments about this work. 
The authors  also thank the referee for his/her
valuable comments and suggestions that have improved this article.
The first author was supported by CNPq/Brazil, and the second author was 
supported by Fapesp/Brazil.

\section{Preliminaries}\label{pre}

In this section we present two well known results that will
be used  in the proofs of our results. The first result give a 
general condition to get a positive  lower bound  to $\lambda_1(M)$
and its proof follows from integration by parts.

\begin{lemma}\label{lem:nicelemma}
Let $M^m$ be a complete Riemannian manifold that carries 
a smooth function $F:M\to\mathbb{R}$ satisfying
\[
\|\grad F\|\leq 1 \quad\text{and}\quad |\Delta F|\geq c,
\]
for some constant $c>0$. Then, for any smooth and relatively compact 
domain $\Omega\subset M$  we have
\[
\lambda_1(\Omega)\geq\frac{c^2}{4},
\]
where $\lambda_1(\Omega)$ is the first eigenvalue of the 
Laplace-Beltrami operator $-\Delta$ in $\Omega$, with Dirichlet
boundary condition.
\end{lemma}

Now, given an isometric immersion $f:M^m\to\widetilde{M}^n$ between 
Riemannian manifolds $M$ and $\widetilde M$, let $\alpha$ denote  
its second fundamental form. Then, the {\em mean curvature vector}  
(not normalized) $H$ of $M$ is  defined by $H=\tr\alpha$.

The second lemma relates the Laplacian of a function 
on $\widetilde M$ and its restriction to $M$ (see, for example,  
\cite[Lemma 2]{gulliver}).

\begin{lemma}\label{lem:gulliver}
Let $f:M^m\to\widetilde M^n$ be an isometric immersion with 
mean curvature vector $H$. Let  $ \widetilde  F:\M\to\mathbb{R}$ be a smooth 
function and let $F= \widetilde F\vert_M$ be its restriction to $M$. 
Then, on $M$, we have:
\begin{eqnarray*}\label{eq:gulliver}
\widetilde\Delta \widetilde F=\Delta F + 
\sum_{i=1}^{n-m}\hs \widetilde F(N_i,N_i) - H(\widetilde F),
\end{eqnarray*}
where $\{N_1,\ldots,N_{n-m}\}$ is an orthonormal frame of 
$TM^{\perp}$.
\end{lemma}


\section{Riemannian Submersions}\label{sub}

Let $\pi:\til M^n\to B^k$ be a Riemannian  submersion of  Riemannian
manifolds. As usual in the literature, given a vector field $X\in\X(B)$ 
we will denote by $\Xt\in\X(\M)$ its unique horizontal lifting. 
In general we use a tilde to denote the lifting to $\til M$ of geometric objects 
in the base $B$.
We  also denote by $\Xt$ the \emph{basics vectors} fields in $\til M$, that is
the vectors fields that are  $\pi$-related to some vector field $X\in\X(B)$.

For  $x\in B$, $\F_x=\pi^{-1}(x)$ denotes the fiber over $x$.
Given  $p\in\F_x$, the differential map $\dd\pi$ 
restricted to the orthogonal subspace $T_p\F_x^\perp$ is an 
isometry onto $T_xB$.
A vector field on $\M$ is called {\em vertical} if it is always tangent
to fibers, and it is called {\em horizontal} if it is always orthogonal to
fibers. Let $\Ver$ denote the {\em vertical distribution} consisting
of vertical vectors and $\Hor$ denote the {\em horizontal distribution}
consisting of horizontal vectors on $M$. The corresponding 
projections from $T\til M$ to $\Ver$ and $\Hor$ are denoted by
the same symbols.

Let $\D\subset T\M$ denote the smooth distribution on $\M$ consisting 
of vertical vectors. The orthogonal distribution $\D^\perp$
is the smooth rank $k$ distribution on $\M$ consisting of horizontal
vectors. The {\em second fundamental form} of the fibers is a symmetric 
tensor $\alpha^\F:\D\times\D\to\D^\perp$, defined by
\[\alpha^\F(v,w)=(\Nt_vW)^\Hor,\]
where $W$ is a vertical extension of $w$.
The {\em mean curvature vector} of the fiber is the 
horizontal vector field $\hf$
defined by $\hf=\tr\alpha^\F$.
In terms of an orthonormal frame, we have
\begin{eqnarray}\label{eq:MeanVector}
\hf(p)= \sum_{i=1}^{n-k}\alpha^\F(e_i,e_i)=
\sum_{i=1}^{n-k}(\Nt_{e_i}e_i)^\Hor,
\end{eqnarray}
where $\{e_1,\ldots,e_{n-k}\}$ is a local orthonormal frame to the fiber 
at $p$. The fibers are {\em minimal} submanifolds of $\M$ when 
$\hf\equiv0$, and are {\em totally geodesic} when $\alpha^\F\equiv0$.

We need some formulas relating the derivatives of $\pi$-related
objects in $\M$ and $B$. Let us start with the divergence of vector
fields. 

\begin{lemma}\label{lem:divfields}
Let $\Xt\in\X(\M)$ be a basic vector field, $\pi$-related to $X\in\X(B)$. 
The following relation holds between the divergence of $\Xt$ and $X$ 
at $x\in B$ and $p\in\F_x$:
\begin{eqnarray*}\label{eq:divfields}
\mathrm{div}\Xt(p)=\mathrm{div} X(x)-\pie\Xt(p),\hf(p)\pid.
\end{eqnarray*}
\end{lemma}
\begin{proof}
Let $\Xt_1,\ldots,\Xt_k,\Xt_{k+1},\ldots,\Xt_n$ be a local orthonormal frame
of $T\til M$, where
$\Xt_1,\ldots,\Xt_k$ are basic fields. The equality follows from
assertions 1 and 3 in \cite[Lemma 1]{oneill}, and
formula (\ref{eq:MeanVector}) using this frame.
\end{proof}

Giving a smooth function $F:B\to\mathbb{R}$
it is easy to see that the gradient of $\til F$ is the horizontal 
lifting of the gradient of $F$, i.e.,
\begin{eqnarray}\label{eq:RelGradiente}
\grad \widetilde F=\widetilde{\grad F}.
\end{eqnarray}

The Laplace operator in $B$ of a smooth function $F:B\to\mathbb{R}$ and the
Laplace operator in $\M$ of its lifting $\widetilde F=F\circ\pi$ are related
by the following formula.

\begin{lemma}\label{lem:Deltarelate}
Let $F:B\to\mathbb{R}$ be a smooth function and set $\widetilde F=F\circ\pi$. Then, for all
$x\in B$ and all $p\in\F_x$:
\begin{eqnarray*}\label{eq:Deltarelate}
\widetilde \Delta \widetilde  F(p)=\Delta F(x)+\pie\grad \widetilde F(p),\hf(p)\pid.
\end{eqnarray*}
\end{lemma}
\begin{proof}
It follows easily from \eqref{eq:RelGradiente} and
Lemma \ref{lem:divfields} applied to the vector fields $\Xt=\grad \widetilde F$
and $X=\grad F$.
\end{proof}

Associated with a Riemannian submersion $\pi:\M\to B$, there are 
two natural $(1,2)-$tensors $T$ and $A$ on $\M$, introduced by
O'Neill in \cite{oneill}, and defined as follows: for vector fields $X$, 
$Y$ tangent to $\M$, the tensor $T$ is defined by
\begin{eqnarray*}
T_XY=\left(\til\nabla_{X^\Ver}Y^\Ver\right)^\Hor + 
\left(\til\nabla_{X^\Ver}Y^\Hor\right)^\Ver.
\end{eqnarray*}
Note that $\pi:\M\to B$ has totally geodesic fibers if and only if
$T$ vanishes identically. The tensor $A$, known as the 
{\em integrability tensor}, is defined by
\begin{eqnarray*}\label{eq:integtensor}
A_XY= \left(\widetilde \nabla_{X^\Hor}Y^\Hor\right)^\Ver + 
\left(\widetilde\nabla_{X^\Hor}Y^\Ver\right)^\Hor.
\end{eqnarray*}
The tensor $A$ measures the obstruction to integrability of the
horizontal distribution $\Hor$. In particular, for any horizontal vector 
field $X$ and any vertical vector field $V$, we have:
\begin{eqnarray}\label{eq:tensorAverhor}
A_XV=\left(\widetilde\nabla_XV\right)^\Hor.
\end{eqnarray}

The following lemma gives useful expressions for the Hessian of the lifting
$\widetilde F:\M\to\mathbb{R}$ of a smooth function $F:B\to\mathbb{R}$,
when we consider horizontal and vertical vector fields. 

\begin{lemma}\label{lemma34}
If $X$ and $Y$ are basic, and $V$ and $W$ are vertical vector fields, 
we have the following expressions for the Hessian of the lifting 
$\til F=F\circ\pi$ of $F$ to $\til M$:
\begin{enumerate}
\item[(a)] $\hs \widetilde F(X,Y)=\hs F\left(\pi_\ast X,\pi_\ast Y\right)\circ\pi$,
\item[(b)] $\hs \widetilde F(V,W)=-\left\pie\alpha^\F(V,W),\grad \widetilde F\right\pid$,
\item[(c)] $\hs \widetilde F(X,V)=-\left\pie A_XV,\grad \widetilde F\right\pid$.
\end{enumerate}
\end{lemma}
\begin{proof}
The first assertion follows from \eqref{eq:RelGradiente} and 
assertion 3 in \cite[Lemma 1]{oneill}. The second 
one is a straightforward calculation, and the third assertion follows 
directly from \eqref{eq:tensorAverhor}.
\end{proof}


\section{Comparison Theorems for Busemann Functions}\label{busemann}

In this section we describe comparison results for the Hessian of  
Busemann functions on two classes of Riemannian manifolds, 
both are generalization of the hyperbolic space. These classes of 
manifolds  will be used as the base space of the Riemannian 
submersions we will consider in our main theorem.

\subsection{Busemann functions on manifolds with bounded 
negative sectional curvature}

Given $a>0$,  let $\mathbb{H}^k(-a^2)$ denote the $k$-dimensional 
hyperbolic space with constant sectional curvature $-a^2$. 
We consider the warped-product model, that is,
$$
\mathbb{H}^k(-a^2) = (\mathbb{R}^{k-1}\times\mathbb{R}, h),
$$
where 
\[
h= e^{-2as}\dd x^2+\dd s^2.
\]
In this model, the curve $\gamma:\mathbb{R}\to \mathbb{H}^k(-a^2) $,
given by $\gamma(s)=(x_0,s)$, is a geodesic for any $x_0\in\mathbb{R}^{k-1}$, 
and the function $\ov F:\mathbb{H}^k(-a^2)\to\mathbb{R}$, given by
\begin{eqnarray}\label{eq:BusemannH}
\ov F(x,s)=s,
\end{eqnarray}
is its associated Busemann function. By a direct 
computation we get
$$
\begin{cases}
\textrm{Hess} \ov F &= e^{-2as}\dd x^2, \\
\Delta \ov F &= (k-1)a.
\end{cases}
$$

Now we will estimate the Hessian of the Busemann function
$\ov F$ defined in a complete Riemannian manifold $B^k$
with sectional curvature between two negative constants. 
In order to obtain the Hessian of $\ov F$, one takes a point $p$ 
on a geodesic sphere of radius $r$, and let the center of the
sphere go to infinity. In this case, the sphere converges to a
horosphere, and the Hessian of the distance function will
converge to the Hessian of the Busemann function. Thus,
a comparison theorem for the Hessian of a Busemann 
function follows from the comparison theorem for the
Hessian of the distance function (see \cite{blps} for a proof).

\begin{lemma}\label{lemma41}
Let $B^k$ be a complete Riemannian manifold with sectional 
curvature $K$ satisfying $-a^2\leq K \leq -b^2$, for some constants 
$a, b>0$. If $\ov F:B\to\mathbb{R}$ is a Busemann function, then 
$$
b\|X\|^2\leq \textrm{Hess}\, \ov F (X,X)\leq a\|X\|^2,
$$
for any  vector $X$ orthogonal to $\grad \overline F$.
\end{lemma}

\subsection{Busemann functions on a class of warped product} 

Let $(N^{k-1},g)$ be a complete Riemannian manifold  and let 
$w :\mathbb{R}\to\mathbb{R}$ be a smooth function. Inspired in the hyperbolic space, 
we consider the Riemannian warped-product manifold 
\begin{eqnarray}\label{eq:prodmanifold}
B=(N\times\mathbb{R},h),
\end{eqnarray}
where
\[
h=e^{2w(s)}g+\dd s^2.
\]

Considere now the Busemann function
$\ov F:B\to\mathbb{R}$ defined by $\ov F(x,s)=s.$
As above, a direct computation gives
$$
\begin{cases}\hs\overline{F} &= w'(s)e^{2w(s)}g, \\
\Delta\ov F &= w'(s)(k-1).
\end{cases}
$$

In particular we have the following lemma:

\begin{lemma}\label{lemma42}
Let $B^k$ be a Riemannian manifold as in (\ref{eq:prodmanifold}) 
and assume that the function $w$ satisfies
$b\leq w' \leq a$, for some constants $a, b>0$. 
If $\ov F:B\to\mathbb{R}$ is the Busemann function defined as above, then 
$$
b\|X\|^2\leq \textrm{Hess}\, \ov F (X,X)\leq a\|X\|^2
$$
for any  vector $X$ orthogonal to $\grad \overline F$.
\end{lemma}

In particular the following consequence will be used in the main theorem.
\begin{corollary} \label{delta}
Under the conditions of Lemma \ref{lemma41} or Lemma \ref{lemma42} 
we have
$$
\Delta \overline F \geq (k-1)b.
$$
\end{corollary}

\remark It is important to point out that Riemannian manifolds given by 
(\ref{eq:prodmanifold}) form a wide class. In particular, we may 
choose the manifold $N$ in such way that $B$ has positive sectional
curvature in some directions (see \cite{veeravalli}).

\section{Main result and examples}\label{main}

In this section, we will apply the previous results in order to get
a lower bound estimates for the first eigenvalue of the Laplace
operator on submanifolds immersed on Riemannian manifolds,
which carries a Riemannian submersion on the two classes of 
manifolds described as before. In particular, using Lemmas 
\ref{lemma41} and \ref{lemma42}, and its corollary above, we 
are able to present a unified proof to both cases. 


\begin{theorem}\label{teo:main2}
Let $B^k$ be a complete Riemannian manifold as in Lemma 
\ref{lemma41} or as in Lemma \ref{lemma42}, and let 
$\pi:\M^n\to B^k$ be a Riemannian submersion. 
Let $M^m$ be a complete Riemannian manifold and let
$f:M^m\to\M^n$ be an isometric immersion. Assume that 
$\ov F:B\to\mathbb{R}$ is a Busemann function and consider its
lifting $\widetilde F:\M\to\mathbb{R}$. If $F= \widetilde F\vert_M$ is 
its restriction to $M$, then
\begin{eqnarray*}\label{eq:estimateNabla}
\Delta F\geq (k-1)b + H^\F(\widetilde F) - 
(n-m)\big(a + 2\|A\|_\infty+\|\alpha^\F\|_\infty\big) + H(\widetilde F).
\end{eqnarray*}
In particular, if 
\begin{eqnarray*}\label{eq:constantc}
c = \inf \{(k-1)b - \|H^\F \| - 
(n-m)\big(a + 2\|A\|_\infty+\|\alpha^\F\|_\infty\big) - \| H\| \}>0,
\end{eqnarray*}
then
$$
\lambda_1(M)\geq \frac{c^2}{4}.
$$
\end{theorem}
\begin{proof}
From Lemma \ref{lem:Deltarelate} and Corollary \ref{delta} we have:
\begin{eqnarray}\label{p1}
\widetilde \Delta \widetilde  F=\ov \Delta \, \ov F+
\pie\grad \widetilde F,\hf \pid
\geq (k-1)b + H^\F(\widetilde F).
\end{eqnarray}
On the other hand, from Lemma \ref{lem:gulliver} we obtain 
\begin{eqnarray}\label{p2}
\widetilde\Delta \widetilde F=\Delta F + 
\sum_{i=1}^{n-m}\hs \widetilde F(N_i,N_i) - H(\widetilde F),
\end{eqnarray}
where $\{N_1,\ldots,N_{n-m}\}$ is an orthonormal frame of 
$TM^{\perp}$. For each $1\leq i\leq n-m$,  we write
$$
N_i=N_i^\Hor+N_i^\Ver,
$$
where $N_i^\Hor$ and $N_i^\Ver$ denote the horizontal and 
vertical projection of $N_i$ onto $T\til M$, respectively.
Moreover, since \eqref{p2} is a tensorial equation, we may assume 
that each $N_i^\Hor$ is basic. 
Thus, using Lemmas \ref{lemma34},  \ref{lemma41} and \ref{lemma42}
we get
\begin{eqnarray*}
\widetilde\Delta \widetilde F \leq \Delta F + 
(n-m)\big( a + 2\|A\|_\infty+\|\alpha^\F\|_\infty\big) - H(\widetilde F).
\end{eqnarray*}
So, plugging this in (\ref{p1}) we obtain
$$
\Delta F \geq (k-1)a + H^\F(\widetilde F) - 
(n-m)\big(b + 2\|A\|_\infty+\|\alpha^\F\|_\infty\big) + H(\widetilde F).
$$
The result follows from Lemma \ref{lem:nicelemma}.
\end{proof}


\subsection{Lower bounds in warped products}\label{exe}

Suppose that the ambient space $\til M^n=\mathbb{H}^k\times_\rho F^{n-k}$
admits a warped product structure, where the warped function $\rho$
satisfies $\|\grad\rho\|/\rho\leq 1$. By considering the projection on the
first factor $\pi:\mathbb{H}^k\times_\rho F^{n-k}\to\mathbb{H}^k$ as a
Riemannian submersion, we have that the tensor $A$ is identically zero,  
$\|\alpha^\F\|_\infty \leq 1$, and in particular $\|H^\F \|\leq n-k$.

Let $M^m$ be a 
complete Riemannian manifold and $f:M^m\to\til M^n$ be an
isometric immersion such that its mean curvature vector $H$
satisfies $\| H\|\leq\alpha$, where $\alpha$ is a positive constant
to be determined . If 
$\ov F:\mathbb{H}^k\to\mathbb{R}$ is the Busemann function given in 
\eqref{eq:BusemannH}, a lower bound estimates for the
infimum in \eqref{eq:constantc} goes as follows:
\begin{eqnarray*}
c &=& \inf\{ k-1-\|H^\F\|-(n-m)(1+\|\alpha^\F\|_\infty)-\|H\| \}  \\
& \geq & \inf\{ k-1-n+k -2(n-m) - \|H\| \} \\
& = & 2(k+m)-3n-1 - \alpha.
\end{eqnarray*}
In particular, $\lambda_1(M)>0$ if we take $0<\alpha < 2(k+m)-3n-1.$

\subsection{Lower bounds in submersions with totally geodesic fibers}
Let $\til M^n$ be a Riemannian manifold with nonpositive sectional
curvature and $\pi:\til M^n\to\mathbb{H}^k$ be a Riemannian submersion
with totally geodesic fibers. This means that $\alpha^\F=0$, and thus
$H^\F=0$. Furthermore, the submersion $\pi$ is integrable in the
sense that the horizontal distribution is integrable 
(cf. \cite[Proposition 3.1]{escobales}). Thus, if $f:M^m\to\til M^n$
is an isometric immersion, whose mean curvature vector $H$
satisfies $\|H\|\leq\alpha$, for some positive constant
$\alpha<k+m-n-1$, we have
\begin{eqnarray*}
c &\geq& k-1-(n-m)- \|H\|  \\
& \geq & k+m-n-1-\alpha >0,
\end{eqnarray*}
and thus $\lambda_1(M)>0$.

\begin{remark}
As suggested by the referee,  the complex hyperbolic space and bounded symmetric domains may be 
interesting examples which are fitted in Theorem \ref{teo:main2}.
\end{remark}

\bibliographystyle{amsplain}

\begin{thebibliography}{99}

\bibitem{cavalcante} P. B\'erard, P. Castillon, M. Cavalcante,\ 
{\em Eigenvalue estimates for hypersurfaces in 
$\mathbb{H}^m\times\mathbb{R}$ and applications}, Pacific J. Math. {\bf 253} 
no. 1, 19--35, (2011). 

 \bibitem{blps} G. P. Bessa, J. H. de Lira, S. Pigola, A. Setti, 
 {\em Curvature Estimates for Submanifolds in Horocylinder},
J. Math. Anal. Appl. {\bf 431}, 1000--1007, (2015).

\bibitem{BM} G. P. Bessa, J. F. Montenegro,\ {\em Eigenvalue 
estimates for submanifolds with locally bounded mean curvature}, 
Ann. Global Anal. Geom. {\bf 24},  279--290, (2003).
 

\bibitem{castillon} P. Castillon, 
{\em Sur l'op\'{e}rateur de stabilit\'{e} des sous-vari\'{e}t\'{e}s \`{a} courbure
moyenne constante dans l'espace hyperbolique,} 
Manuscripta Math. {\bf 94}, 385--400, (1997).

\bibitem{cheung} L.-F. Cheung, P.-F. Leung,\ {\em Eigenvalue 
estimates for submanifolds with boundary mean curvature in 
the hyperbolic space}, Math. Z. {\bf 236}, 525--530, (2001).

\bibitem{gulliver} J. Choe, R. Gulliver,\ {\em Isoparametric inequalities 
on minimal submanifolds of space forms},
Manuscripta Math. {\bf 77}: 2-3, 169--189, (1992).

\bibitem{escobales} R. H. Escobales, Jr.,\ {\em Riemannian
submersions with totally geodesic fibers}, J. Differential Geom. {\bf 10}, 253--276, (1975).

\bibitem{mckean} H. P. McKean,\ {\em An upper bound to the spectrum
of $\Delta$ on a manifold of negative curvature}, J. Differential Geom.
{\bf 4}, 359--366, (1970).

\bibitem{oneill} B. O'Neill,\ {\em The fundamental equations of a submersion},
Michigan Math. J. \textbf{13}, 459--469, (1966),.

\bibitem{yau} R. Schoen, S.T. Yau,\ {\em Lectures on Differential Geometry}, 414 p.,
International Press, Cambridge, MA (2010).

\bibitem{veeravalli} A. R. Veeravalli,\ {\em Une remarque sur l'in\'egalit\'e de McKean}, Comment. Math. Helv. {\bf 78}, 884--888, (2003).

\end{thebibliography}

\end{document}